\documentclass{article}
\usepackage[utf8]{inputenc}
\usepackage{amsmath}
\usepackage{amssymb}
\usepackage{amsthm}

\usepackage{times}
\usepackage{geometry}
\usepackage[utf8]{inputenc}
\usepackage[T1]{fontenc}
\usepackage{tikz-cd}
\usepackage{hyperref}
\hypersetup{
	colorlinks,
	citecolor=blue,
	filecolor=black,
	linkcolor=black,
	urlcolor=magenta
}

\newtheorem{thm}{Theorem}[section]

\newtheorem{lem}[thm]{Lemma}

\newtheorem*{thm*}{Theorem}

\newtheorem{prop}[thm]{Proposition}

\theoremstyle{definition}
\newtheorem{defin}[thm]{Definition}

\newtheorem*{remark}{Remark}
\newcommand{\Hcountable}{\operatorname{H}_{\omega_1}}
\newcommand{\Hkappa}{\operatorname{H}_{\kappa}}

\newcommand{\Hcontinuum}{\operatorname{H}_{\mathfrak{c}^+}}

\newcommand{\ABE}{
	\operatorname{(A,B,E)}}
\newcommand{\ABEq}{
	\operatorname{(A_1,B_1,E_1)}}
\newcommand{\ABEqq}{
	\operatorname{(A_2,B_2,E_2)}}

\newcommand{\ABEE}{
	\operatorname{\langle A,B,E\rangle}}

\newcommand{\diamondABE}{\diamondsuit\operatorname{(A,B,E)}}
\newcommand{\diamondstrABEN}{\diamondsuit_{\operatorname{STR}}
	\operatorname{((A,B,E),N)}}
\newcommand{\diamondelABEN}{\diamondsuit_{\operatorname{EL}}\operatorname{((A,B,E),N)}}

\newcommand{\diamondstrABEcountable}{\diamondsuit_{\operatorname{STR}}
	\operatorname{(A,B,E)}}
\newcommand{\diamondelABEcountable}{\diamondsuit_{\operatorname{EL}}\operatorname{(A,B,E)}}

\newcommand{\diamondstrABEcontinuum}{\diamondsuit_{\operatorname{STR}}
	\operatorname{((A,B,E),\Hcontinuum)}}
\newcommand{\diamondelABEcontinuum}{\diamondsuit_{\operatorname{EL}}\operatorname{((A,B,E),\Hcontinuum)}}

\newcommand{\diamondstrABEkappaLL}{\diamondsuit_{\operatorname{STR}}
	\operatorname{((A,B,E),(\Hkappa;\LL))}}
\newcommand{\diamondelABEkappaLL}{\diamondsuit_{\operatorname{EL}}\operatorname{((A,B,E),(\Hkappa; \LL))}}

\newcommand{\phiABE}{\Phi
	\operatorname{(A,B,E)}}
\newcommand{\phicABE}{\Phi
	\operatorname{([A]^{\le\omega},B,E)}}

\newcommand{\diamondstareldH}
{\diamondsuit^*_{\operatorname{EL}}(\mathfrak{d},(\Hcountable; \in))}

\newcommand{\diamondelbH}{\diamondsuit_{\operatorname{EL}}(\mathfrak{b},(\Hcountable; \in))}
\newcommand{\diamondeldH}{\diamondsuit_{\operatorname{EL}}(\mathfrak{d},(\Hcountable; \in))}
\newcommand{\diamondelrH}{\diamondsuit_{\operatorname{EL}}(\mathfrak{r},(\Hcountable; \in))}

\newcommand{\diamondstarstrnonmeagre}{\diamondsuit^*_{\operatorname{STR}}(\operatorname{non}(\mathcal{M}), (\Hcountable; \in))}
\newcommand{\diamondstarstrnonnull}{\diamondsuit^*_{\operatorname{STR}}(\operatorname{non}(\mathcal{N}), (\Hcountable; \in))}

\newcommand{\diamondelnonmeagreH}{\diamondsuit_{\operatorname{EL}}(\operatorname{non}(\mathcal{M}),(\Hcountable; \in))}

\newcommand{\gameABE}{\mathcal{G}_{\operatorname{(A,B,E)}}(\mathcal{H},\mathcal{B})}

\newcommand{\LL}{\mathcal{L}}
\newcommand{\zfcm}{\operatorname{ZFC}^-}

\newcommand{\bad}{\operatorname{BAD}}
\newcommand{\good}{\operatorname{GOOD}}

\title{Guessing genericity -- looking at parametrized diamonds from a different perspective}
\author{Ziemowit Kostana\footnote{The author was supported by the European Research Council (grant agreement ERC-2018-StG 802756), and by the GAČR project EXPRO 20-31529X, RVO: 67985840.}\\
	Institute of Mathematics CAS, Prague\\
	Bar-Ilan University, Ramat Gan}
\date{}

\begin{document}
	
	\maketitle
	\begin{abstract}
		We introduce and study a family of axioms that closely follows the pattern of parametrized diamonds, studied by Moore, Hru\v{s}\'ak, and D\v{z}amonja in \cite{mhd}. However, our approach appeals to model theoretic / forcing theoretic notions, rather than pure combinatorics. The main goal of the paper is to exhibit a surprising, close connection between seemingly very distinct principles. As an application, we show that forcing with a measure algebra preserves (a variant of) $\diamondsuit(\mathfrak{d})$, improving an old result of M. Hru\v{s}\'ak.
	\end{abstract}
	

	\section{Introduction}
	
	The aim of this paper is to introduce an alternative approach to parametrized diamond principles. These axioms  are sufficient to construct numerous combinatorial objects, whose existence was originally known to follow from $\diamondsuit$. A common catchphrase is that parametrized diamonds measure a "degree" of $\diamondsuit$ in a similar manner as cardinal invariants of the real line measure "how much" of $CH$ is true. 
	
	The first axiom in this spirit, the \emph{weak diamond}, was studied by Devlin and Shelah in \cite{devlinshelah}, and turned out to be equivalent to a cardinal arithmetic statement $2^\omega<2^{\omega_1}$. Much later, a similar axiom $\diamondsuit_\mathfrak{d}$ was introduced in \cite{hrusak} in order to partially answer the question whether $\mathfrak{d}=\omega_1$ implies $\mathfrak{a}=\omega_1$. These two examples served as a ground for developing a more general theory in \cite{mhd}. These axioms, the \emph{parametrized diamonds}, have found numerous applications, for example: $\diamondsuit(\mathfrak{b})$ gives an maximal almost disjoint family of size $\omega_1$, $\diamondsuit(\operatorname{non}(\mathcal{M}))$ implies there is a Suslin tree, and $\diamondsuit(\mathfrak{r})$ gives a $\omega_1$-generated ultrafilter on $\omega$ (\cite{mhd}). Different variants of parametrized diamonds were also applied in, for example, \cite{union} and \cite{idealindependent} in connection with combinatorics of $[\omega]^\omega$, and in \cite{top1}, \cite{top2}, \cite{top3} for topological constructions. A certain different axiom, similar in the spirit, though not following the very same pattern, was applied in \cite{calderonfarah}.

	\section{Preliminaries}
	We denote the ideals of meagre and null subsets of reals by $\mathcal{M}$, and $\mathcal{N}$ respectively. For an uncountable cardinal $\kappa$, $\Hkappa$ is the family of all sets whose transitive closure has size $<\kappa$. We fix a well ordering $\prec^*$ of $\Hcontinuum$. 
	For $\kappa\ge \omega$ we denote by $\mathbb{B}_\kappa$ the forcing associated with the measure algebra on $2^\kappa$, and by $\mathbb{C}_\kappa$, the Cohen forcing on $2^\kappa$, that is
	$$\mathbb{C}_\kappa=\{p:\operatorname{dom}(p)\rightarrow2|\; \operatorname{dom}(p)\in [\kappa]^{<\omega}\}.$$
	We write $N\prec M$ for \emph{"$N$ is an elementary substructure of $M$"}. We write $s \sqsubseteq t$ for \emph{"$s$ is an initial segment of a sequence $t$"}, and $s \sqsubset t$ for \emph{"$s$ is a proper initial segment of a sequence $t$"}.
	
	Suppose $\LL$ a language extending $\{ \in \}$. A \emph{chain} of elementary submodels of $(\Hkappa; \LL)$ is a sequence of the form $\langle M_\alpha \mid \alpha<\lambda \rangle$, where $(M_\alpha; \LL) \prec (\Hkappa; \LL)$ for every $\alpha$, and $M_\beta = \bigcup_{\alpha<\beta} M_\alpha$ for every limit $\beta<\lambda$. When $\LL$ is clear from the context, we sometimes omit it.
	
	\section{Abstract invariants of the continuum}
	
	\begin{defin}[\cite{vojtas}] \label{invartiant}
		A triple $\ABE$ is called an \emph{invariant}, if 
		\begin{enumerate}
			\item $A,B \in \Hcontinuum$,
			\item $E \subseteq A \times B$,
			\item $\forall \; a \in A \; \exists \; b \in B \quad (a,b)\in E$,
			\item $\forall \; b \in B \; \exists \; a \in A \quad (a,b)\notin E$.
		\end{enumerate}
	\end{defin}
	We write $a\, E\, b$ instead of $(a,b) \in E$.
	
	\begin{defin}\label{sigma}
		If $(A,B,E)$ is an invariant, we define a associated invariant $([A]^{\le \omega},B,\sigma E)$ in which $\sigma E$ is defined by the condition
		$$a \; \sigma E \;b \iff \forall \; x \in a \quad x \, E \, b.$$
	\end{defin}
	
	\begin{defin}
		If $\ABE$ is an invariant, then its \emph{evaluation} is the cardinal
		$$
		\ABEE = \min\{|X|\mid X \subseteq B, \; \forall \; a \in A \; \exists \; x \in X \quad a\, E\, x \}.$$
	\end{defin}
	
	We will be mostly interested in the invariants from the following list, where $\mathcal{I} \in \{\mathcal{M},\mathcal{N}\}$. By $<^*$ we denote the order of eventual domination on $\omega^\omega$, i.e. $$f<^*g \iff \exists \; n<\omega \; \forall\; k>n \;\; f(k)<g(k).$$
	\begin{itemize}
		\item $\operatorname{cov}(\mathcal{I})=(\mathbb{R},\mathcal{I},\in)$,
		\item $\operatorname{non}(\mathcal{I})=(\mathcal{I},\mathbb{R},\not\ni)$,
		\item $\operatorname{cof}(\mathcal{I})=(\mathcal{I},\mathcal{I},\subseteq)$,
		\item $\operatorname{add}(\mathcal{I})=(\mathcal{I},\mathcal{I},\not\supseteq)$,
		\item $\mathfrak{b}=(\omega^\omega,\omega^\omega,\not>^*)$,
		\item $\mathfrak{d}=(\omega^\omega,\omega^\omega,<^*),$
		\item $\mathfrak{r}=(2^\omega,[\omega]^\omega,``\text{ restriction is eventually constant}").$
	\end{itemize}

	\begin{defin}[\cite{mhd}]
		If $\ABE$ is an invariant, we denote by $\phiABE$ the following sentence:
		$$\forall \; F:2^{<\omega_1}\rightarrow A \quad \exists \; g:\omega_1\rightarrow B \quad \forall \; f \in 2^{\omega_1}$$
		the set
		$$\{\alpha< \omega_1 \mid F(f\restriction \alpha)\;E\; g(\alpha)\}$$
		is stationary.
	\end{defin}
	
	To put $\phiABE$ into some context, $\phiABE$ is a consequence of $\diamondsuit$, for any invariant $\ABE$, and $\diamondsuit$ is equivalent to $\Phi(\mathbb{R},\mathbb{R},=)$. Also, the \emph{weak diamond} introduced in \cite{devlinshelah} is just $\Phi(2,2,=)$ (and is equivalent to $2^\omega<2^{\omega_1}$). However, a restricted variant of $\phiABE$ proved more fruitful in applications.
	
	\begin{defin}
		An invariant $\ABE$ is \emph{Borel} if $A,B$ are Polish spaces, and $E$ is a Borel subset of $A\times B$. A function $F:2^{<\omega_1}\rightarrow A$ is Borel, if every restriction $F\restriction 2^\alpha$, for $\alpha<\omega_1$ is Borel.
	\end{defin}
	
	\begin{defin}
		If $\ABE$ is a Borel invariant, by $\diamondABE$ we denote the following sentence:\\
		For every Borel function $F:2^{<\omega_1}\rightarrow A$, there exists $g:\omega_1\rightarrow B$ such that for all $f \in 2^{\omega_1}$,\\
		the set
		$$\{\alpha< \omega_1 \mid F(f\restriction \alpha)\; E\; g(\alpha)\}$$
		is stationary.
	\end{defin}

	\subsection{Fragments of set theory}
	
	In the forthcoming sections $N$ will always denote a transitive model of $\zfcm$, that is $\operatorname{ZFC}$ with the powerset axiom excluded, such that moreover $N \supseteq \Hcountable$. Only $N \subseteq \Hcontinuum$ will be relevant for us, but we don't need to assume this explicitly.
	
	Assume $N=(N;\LL)$ is such model in a relational language $\LL \supseteq \{\in\}$.  We denote by $\theta=\theta[(N;\LL)]$ the $\emph{truth predicate}$ for $(N;\LL)$, i.e. a relation such that for any formula $\phi$ in the language $\LL$ with $n$ free variables, and any $x_1,\ldots,x_n \in N$
	$$\theta(\phi,x_1,\ldots,x_n) \iff (N;\LL) \models \phi(x_1,\ldots,x_n).$$
	It is standard to verify that in the structure $(N;\LL,\theta)$ we can write the formula 
	$$\phi(m)="\text{\emph{$m$ is a countable elementary submodel of $(N;\LL)$}}".$$
	
	We also want to be able to write the formula "$m$ is an elementary submodel of $(N;\LL,\theta)$''. For this reason, we iterate the procedure, and define $\theta_{n+1}$ to be a truth predicate for $(N;\; \LL, \theta_0,\ldots\theta_n)$, where $\theta_0=\theta$.

	\begin{defin} An invariant $\ABE$ is \emph{definable over $N$} if the sets $A,\, B$, and $E$ are definable with parameters subsets of $(N;\in)$.
	\end{defin}
	Note that, by definition, every invariant is definable over $(\Hcontinuum; \in)$. Finally, we fix a bijection $\Sigma:2^\omega \rightarrow \Hcountable$, definable over $(\Hcountable; \in)$.

	\section{Games $\gameABE$ on transitive structures}
	
	\begin{defin}
		Let $\ABE$ be an invariant. Suppose that $\mathcal{H}=\langle h_\alpha \mid \alpha<\omega_1\rangle$ is a sequence of parameters from a model $N$, and $\mathcal{B}=\langle b_\alpha \mid  \alpha<\omega_1\rangle \subseteq B$. The game $\gameABE$ on $N$ is defined by the following rules:
		\begin{enumerate}
			\item in an even step $\alpha$, $\bad$ plays $\ast_\alpha \in \Hcountable$,
			\item in an odd step $\alpha+1$, $\good$ plays a countable elementary submodel $M_{\alpha+1}\prec N$, such that
			$$\{\ast_\alpha,h_\alpha, \langle M_\beta \mid \beta<\alpha, \text{ $\beta$ is odd}\rangle \} \in {M_{\alpha+1}},$$
			\item if $\beta$ is limit, then $\ast_\alpha$ must be definable from $\langle\ast_{2\cdot{\alpha}}\mid \alpha<\beta\rangle$ in the structure $N$, without additional parameters.
		\end{enumerate}
		
		$\good$ wins, if after all $\omega_1$ many steps the following set is stationary
		$$\{\alpha<\omega_1 \mid (A\cap M_{\alpha+1})\,\sigma E\, b_\alpha\}.$$
	\end{defin}
	
	Let us give some explanation of the underlying ideas of this definition. In the applications, the relation 
	$$(A\cap M_{\alpha+1})\; \sigma E\; b_\alpha$$
	will usually mean that $b_\alpha$ is a specific type of a generic real over $M_{\alpha+1}$ -- for example a Cohen, or a random real. It will be convenient to have a name for sequences satisfying 3. 
	
	\begin{defin}
		For a model $N$, a sequence $\langle x_\alpha \mid \alpha<\omega_1 \rangle$ is \emph{continuous} if for every limit $\beta<\omega_1$, $x_\beta$ is definable inside $N$ with $\langle x_{\alpha} \mid \alpha<\beta \rangle$ as the only parameter\footnote{Strictly speaking, $\ast_\alpha$ might be undefined for some $\alpha$, and in this case we make convention that $\ast_\alpha = \emptyset$.}
	\end{defin}
	
	The definition mentions two sequences -- $\mathcal H$ and $\langle \ast_{2 \cdot \alpha} \mid \alpha<\omega_1\rangle$. Why we need both of them? The sequence $\mathcal{H}$ plays a different role in two ways: firstly, $\mathcal{B}$ will depend on $\mathcal{H}$, whilst it will be independent of $\langle\ast_\alpha \mid \alpha<\omega_1\rangle$. Secondly, $\mathcal{H}$ doesn't have to be continuous. The continuity of parameters $\ast_\alpha$ is however necessary for $\good$ to have a chance to win: otherwise $\bad$ can just play $\ast_\alpha = b_\alpha$ every time. Then it is not possible that $b_\alpha$ is a generic real over $M_{\alpha+1}$, since $b_\alpha=\ast_\alpha \in M_{\alpha+1}$. 
	
	If $N=\Hcontinuum$, we can simplify the definition, by putting already the full sequence $\mathcal{H}$ into $M_1$. But in general $\mathcal{H}$ doesn't need to be an element of $N$, so the models $M_\alpha$ must cover it gradually. 
	
	We list some classical invariants that correspond to specific types of genericity.
	
	\begin{tabular}{ l | l}
		$\ABE$ & the relation \; $(A\cap M)\; \sigma E\; b$ \\
		\hline
		$\operatorname{non}(\mathcal{M})$ & $b$ is a Cohen real over $M$  \\
		$\operatorname{non}(\mathcal{N})$ & $b$ is a random real over $M$ \\
		$\mathfrak{d}$ & $b$ is a dominating real over $M$\\
		$\mathfrak{b}$ & $b$ is an unbounded real over $M$\\
		$\mathfrak{r}$ & $b$ is a reaping real over $M$\\
		$\operatorname{cov}(\mathcal{M})$ & $b$ is a meagre set such that $\mathbb{R} \cap M \subseteq b$\\
		$\operatorname{cov}(\mathcal{N})$ & $b$ is a null set such that $\mathbb{R} \cap M \subseteq b$\\
		$([\mathbb{R}]^\omega,\mathbb{R},\not=)$ & $b \in \mathbb{R}\setminus M$\\
		$(\mathbb{R},[\mathbb{R}]^\omega,\in)$ & $b\subseteq \mathbb{R}$ is countable, and $\mathbb{R}\cap M \subseteq b$\\
	\end{tabular}\\

	So, informally speaking, $\good$ is trying to ensure that the models $M_{\alpha+1}$ are:
	\begin{enumerate}
		\item Big enough to contain $\mathcal{H}$, and whatever $\bad$ wants them to contain,
		\item Small enough, so that the reals $b_\alpha$, \underline{given in advance}, are generic over the corresponding models.
	\end{enumerate}
	We will now introduce a principle saying that there is always a sequence of $b_\alpha$'s, so that $\good$ can win.
	\subsection{Strategic diamond}
	\begin{defin} Based on the preceding definition, we introduce the following combinatorial principles:
		\begin{itemize}
			\item $\diamondstrABEN$ says: 
			For each sequence of parameters $\mathcal{H}=\langle h_\alpha \mid \alpha<\omega_1\rangle \subseteq N$, there exists a sequence $\mathcal{B} = \langle b_\alpha \mid \alpha<\omega_1 \rangle \subseteq B$ such that $\good$ has a winning strategy in $\gameABE$ on $N$.
			
			\item $\diamondsuit_{\operatorname{STR}}^*{\operatorname{((A,B,E),N)}}$ says:     
			For each sequence of parameters $\mathcal{H}=\langle h_\alpha \mid \alpha<\omega_1\rangle \subseteq N$, there exists a sequence $\mathcal{B} = \langle b_\alpha \mid \alpha<\omega_1 \rangle \subseteq B$ such that $\good$ has a winning strategy in $\gameABE$ on $N$, ensuring moreover that the set
			$$\{\alpha<\omega_1 \mid (A\cap M_{\alpha+1})\; \sigma  E\; b_\alpha\}$$
			contains a club.
			\item The \emph{strategic diamond} $\diamondstrABEcountable$ says that $$\diamondsuit_{\operatorname{STR}}{\operatorname{((A,B,E),(\Hcountable;\in,\prec^*,\theta_0,\ldots,\theta_n)}}$$ holds for every integer $n$, where
			$\theta_0$ is a truth predicate for $(\Hcountable; \in,\prec^*)$, and $\theta_{n+1}$ is a truth predicate for $(\Hcountable;\in,\prec^*,\theta_0,\ldots,\theta_n))$.
		\end{itemize}
	\end{defin}
	
	It might be the case that $\bad$ plays consecutive elements of some continuous sequence $\langle\ast_\alpha \mid \alpha<\omega_1\rangle$ given in advance, regardless of the plays of $\good$. This situation gives rise to a principle we call \emph{elementary diamond}. In fact, at the current stage of research, there is no visible advantage of strategic diamond over a simpler -- and seemingly weaker -- elementary diamond, which we will now introduce.  
	\subsection{Elementary diamond}
	\begin{defin} We introduce the following principles:
		\begin{itemize}
			\item $\diamondelABEN$: 
			For every given sequence of parameters $\mathcal{H} = \langle h_\alpha \mid \alpha<\omega_1 \rangle \subseteq N$, we can choose a sequence $\mathcal{B} = \langle b_\alpha \mid \alpha<\omega_1 \rangle \subseteq B$, such that for any continuous sequence $\langle \ast_\alpha \mid \alpha<\omega_1\rangle \subseteq \Hcountable$ there exists a stationary set of indices $\beta<\omega_1$ for which exists a countable $M_\beta\prec N$ that satisfies:
			\begin{itemize}
				\item $\{\ast_\beta, \langle h_\alpha \mid \alpha \leq \beta \rangle \} \in M_\beta$,
				\item $(A\cap M_{\beta}) \; \sigma E \; b_\beta.$
				
			\end{itemize}
			
			\item $\diamondsuit_{\operatorname{EL}}^*{\operatorname{((A,B,E),N)}}$ is like $\diamondelABEN$, except that $M_\beta$ exists for a club set of $\beta<\omega_1$.
			
			\item The \emph{elementary diamond} $\diamondelABEcountable$ says that $$\diamondsuit_{\operatorname{EL}}{\operatorname{((A,B,E),(\Hcountable;\in,\prec^*,\theta_0,\ldots,\theta_n)}})$$ holds for every integer $n$, where
			$\theta_0$ is a truth predicate for $(\Hcountable; \in,\prec^*)$, and $\theta_{n+1}$ is a truth predicate for $(\Hcountable;\in,\prec^*,\theta_0,\ldots,\theta_n)$.
		\end{itemize}
	\end{defin}
	
	We will see that there is no loss of generality in assuming that the sets $\ast_\alpha$ are of the form $\ast_\alpha = X\cap \alpha$, for some $X \subseteq \omega_1$. The relations between $\diamondelABEN$, $\diamondstrABEN$, and instances of $\Phi\ABE$ will heavily depend on the structure $N$. 
	
	\begin{prop}
		For any invariant $\ABE$, and any model $N$, the following holds:
		$$\diamondstrABEN \iff \diamondsuit_{\operatorname{STR}}
		\operatorname{(([A]^{\le \omega},B, \sigma E),N)},$$
		$$\diamondelABEN \iff \diamondsuit_{\operatorname{EL}}
		\operatorname{(([A]^{\le \omega},B,\sigma E),N)}.$$
	\end{prop}
	\begin{proof}
		It is sufficient to show that for any countable $M\prec N$, and $b \in B$, the relation
		$$ A \cap M \, E\, b$$
		implies
		$$([A]^\omega \cap M) \; \sigma E \; b.$$
		To see this, pick any $a \in [A]^\omega \cap M$. Since $a$ is countable, $a \subseteq M$, and so $a \, E\, b$. \qedhere
	\end{proof}
	
	\subsection{$\gameABE$ on $\Hcontinuum$}
	
	In the simplest case, both strategic and elementary diamonds are reformulations of $\phiABE$. In this section we write $\Hcontinuum$ for $(\Hcontinuum; \in)$.
	
	\begin{thm} \label{mainthm}
		The following conditions are equivalent for any invariant $\ABE$:
		\begin{enumerate}
			\item $\diamondstrABEcontinuum$,
			\item $\diamondelABEcontinuum$,
			\item $\diamondelABEcontinuum$, restricted to the case where 
			$\ast_\alpha=X \cap \alpha,$ for some $X\subseteq \omega_1$,
			\item $\phicABE$.
		\end{enumerate}
	\end{thm}
	\begin{proof}

		$1.\implies 2.$ Fix a sequence $\mathcal{H}$, and let $\mathcal{B}$ be such that $\good$ has a winning strategy in $\gameABE$. Let $\langle \ast_\alpha \mid \alpha<\omega_1\rangle$ be arbitrary. Consider a run of the game $\gameABE$ on $\Hcontinuum$, where $\bad$ plays the sets $\ast_\alpha$, regardless of the plays of $\good$. Assume that $\good$ plays according to a winning strategy, and let $\langle M^*_\alpha\mid \alpha<\omega_1 \rangle$ be the resulting chain of elementary submodels. It is straightforward that $\langle M_\alpha \mid \alpha< \omega_1\rangle$ witnesses $\diamondelABEcontinuum$, where $M_\alpha := M^*_{\alpha+1}$.\\
		
		$2. \implies 3.$ Clear.\\
		
		$3. \implies 4.$ Let $F:2^{<\omega_1}\rightarrow A$ be any function. Let $\mathcal{B}$ be a sequence obtained from $3.$ for $\mathcal H=\langle F\rangle$. Finally, let $f \in 2^{\omega_1}$ be given. Using $3.$ we can find a stationary set of indices $\beta$, together with $M_\beta \prec \Hcontinuum$, satisfying:
		\begin{itemize}
			\item $f\restriction \alpha,\, F\in M_\beta$,
			\item $(A \cap M_\beta) \; \sigma E \; b_\beta.$
		\end{itemize}
		But the latter bullet ensures, in particular, that $F(f\restriction \alpha) \; E \; b_\alpha$. This shows that $\mathcal{B}$ witnesses $\phiABE$.\\
		
		$4. \implies 1.$ 
		Let $\mathcal{H}$ be a given sequence of parameters. We define the suitable function $F:2^{<\omega_1}\rightarrow [A]^\omega$ by the induction: suppose that $F\restriction 2^\beta$ is defined for all $\beta<\alpha$. For $s \in 2^{<\omega_1}$ let us define (recursively) $M(s)$ as a countable elementary submodel of $\Hcontinuum$, that contains $s$, $\mathcal{H}$, and the set $$\{M(s')|\; s' \sqsubset s\}.$$
		Given $s \in 2^\alpha$, we put $F(s)=A\cap M(s)$.
		
		Let $\langle b_\alpha \mid \alpha<\omega_1\rangle$ be a guessing sequence for $F$. We describe the winning strategy for $\good.$ The strategy will consist of building a chain of models $M_\alpha$ together with functions $f_{\alpha} \in 2^{<\omega_1}$, coding the sets $\ast_\alpha$. Suppose we want to define $M_{\alpha+1}$, and $f_\alpha$.
		\begin{itemize}
			\item If $\alpha=0$, we put $f_\alpha=\emptyset$.
			\item If $\alpha$ is successor, then $\bad$ just played $\ast_{\alpha}$, where $\alpha=\beta+2$. We set $f_{\alpha}=f_\beta^\frown t_{\beta+2}$, where $t_{\beta+2} \in 2^\omega$, and $\ast_{\alpha}=\Sigma(t_{\beta+2})$.
			\item If $\alpha$ is limit, we put $f_\alpha=\displaystyle{\bigcup_{\beta<\alpha}f_\beta}$. Since $\ast_\alpha$ is definable from $\langle \ast_\beta|\; \beta<\alpha\rangle$, it is also definable from $f_\alpha$. 
		\end{itemize}
		In both cases, we put $M_{\alpha+1}=M(f_\alpha)$. Notice that for all $\alpha< \omega_1$, $f_{\alpha \cdot 2} \in 2^{\omega\cdot \alpha}$, and for $\alpha$ limit, $f_\alpha \in 2^{\omega \cdot \alpha}$. Why is this a correct strategy?
		\begin{itemize}
			\item $h_\alpha \in M_{\alpha+1}=M(f_\alpha)$, since both $\alpha$ and $\mathcal{H}$ are elements of $M(f_\alpha)$.
			\item $\ast_\alpha \in M_{\alpha+1}$ since, as we noticed, $\ast_{\alpha}$ is definable from $f_\alpha$.
			\item $\{M_\beta|\; \beta<\alpha, \text{ $\beta$ is odd}\} \in M(f_\alpha)$, because $\{M_\beta|\; \beta<\alpha, \text{ $\beta$ is odd}\}$ is definable from $\{M(s)|\; s \sqsubset f_\alpha\}\in M(f_\alpha).$
		\end{itemize}
		Finally, why is this a winning strategy? Let $f = \displaystyle{\bigcup_{\alpha<\omega_1}f_\alpha}$. For club many $\alpha<\omega_1$, it is the case that $f_\alpha \in 2^\alpha$, so $M_{\alpha+1}=M(f_\alpha)=M(f\restriction \alpha)$. For stationarily many of these $\alpha$ it is the case that $F(f\restriction \alpha) \, E\, b_\alpha$, but $F(f\restriction \alpha)=A\cap M(f\restriction \alpha)=A \cap M_{\alpha+1}$. \qedhere
	\end{proof}
	\subsection{$\gameABE$ on $\Hkappa$}
	
	In the case when $N=\Hkappa$ for arbitrary regular $\kappa$, we were not able to reproduce the previous result. We can still prove that the principles $\diamondelABEN$ and $\diamondABE$ are closely related, while the apparent mismatch comes from different restrictions on the function $F$.

	\begin{thm} \label{mainthm2}
		Let $\ABE$ be an invariant definable (with parameters) over $(\Hkappa;\LL)$, where $\LL$ is a countable language extending $\{\in, \prec^*\}$. Then each of the following conditions implies the next:
		\begin{enumerate}
			\item $\phicABE$, restricted to functions $F$, such that every restriction $F\restriction 2^\alpha$ is definable over the structure $(\Hkappa;\LL,\theta)$, where $\theta=\theta[(\Hkappa;\LL)]$ is a truth predicate,
			\item $\diamondstrABEkappaLL$,
			\item $\diamondelABEkappaLL$, restricted to the case where $\ast_\alpha=X \cap \alpha,$ for some $X\subseteq \omega_1$,
			\item $\diamondelABEkappaLL$,
			\item $\phicABE$, restricted to functions $F$, such that every restriction $F\restriction 2^\alpha$ is definable over the structure $(\Hkappa;\LL)$.
		\end{enumerate}
	\end{thm}

	\begin{proof}[Proof of Theorem \ref{mainthm2}] 
		$1. \implies 2.$ Let $\mathcal{H}=\langle h_\alpha\mid \alpha<\omega_1\rangle$ be a given sequence of parameters. We define the suitable function $F:2^{<\omega_1}\rightarrow A$ by the induction: suppose that $F\restriction 2^\beta$ is defined for all $\beta<\alpha$. For $s \in 2^{<\omega_1}$ let us define (recursively) $M(s)$ as the $\prec^*$-least countable elementary submodel of $(\Hkappa; \LL)$, that contains $s$, $\langle h_\beta|\;\beta<|s|\rangle$, and the set $$\{M(s')|\; s' \sqsubset s\}.$$
		Given $s \in 2^\alpha$, we put $F(s)=A \cap M(s)$.
		
		Notice that the family $\{M(s) \mid s \in 2^{<{\omega_1}}\}$ is definable over $(\Hkappa;\LL,\theta)$. Let $\langle b_\alpha \mid \alpha<\omega_1\rangle$ be a guessing sequence for $F$. The winning strategy for $\good$ is defined exactly as in Theorem \ref{mainthm}, but this time we make use of definability of the coding function $\Sigma$. A straightforward verification shows that every function of the form $F\restriction 2^\alpha$ is definable over $(\Hkappa; \LL,\theta)$.\\
		
		$2. \implies 3.$ Straightforward, like $1. \implies 2.$ in Theorem \ref{mainthm}.\\
		
		$3. \implies 4.$ Since $\langle \ast_\alpha \mid \alpha < \omega_1 \rangle$ is continuous, we can find $X \subseteq \omega_1$ such that $\ast_\alpha$ is definable in $(\Hcountable; \in)$ from $X\cap \alpha$, for club many $\alpha<\omega_1$.\\
		
		$4. \implies 5.$ Similar to $3.\implies 4.$ in Theorem \ref{mainthm}. The only difference is that $h_\alpha$ is a set of parameters, such that $F\restriction 2^\alpha$ is definable inside $(\Hkappa;\LL)$ from $h_\alpha$. \qedhere \end{proof}
	
	The most important case is $\kappa=\omega_1$. In this case we can apply Theorem \ref{mainthm2} to obtain an equivalent definition of $\phiABE$ for a class of functions that can be coded by real numbers, but in a much broader sense than Borel functions.
	
	\begin{thm}\label{mainthm3}
		Let $\ABE$ be an invariant definable over $(\Hcountable;\in,\prec^*)$. Then the following conditions are equivalent:
		\begin{enumerate}
			\item $\phicABE$, restricted to those functions $F$ for which there exists $n<\omega$ such that every restriction $F\restriction 2^\alpha$ is definable over the structure $(\Hcountable;\in,\prec^*,\theta_0,\ldots,\theta_n)$,
			\item $\diamondstrABEcountable$,
			\item $\diamondelABEcountable$, restricted to the case where $\ast_\alpha=X \cap \alpha,$ for some $X\subseteq \omega_1$,
			\item $\diamondelABEcountable$,
		\end{enumerate}
	\end{thm}
	\begin{proof}
		Follows from Theorem \ref{mainthm2}, applied to $\kappa =\omega_1$, and $\LL=\{\in,\prec^*,\theta_0,\ldots,\theta_n\}$, for varying $n<\omega$. \qedhere
	\end{proof}
	
	\subsection{Tukey connections}
	
	We can compare invariants with respect to \emph{Tukey connections}.
	\begin{defin}[\cite{vojtas}]
		If $\ABEq$, $\ABEqq$ are invariants, a Tukey connection from $\ABEq$ to $\ABEqq$ is a pair of mappings $\phi:A_1\rightarrow A_2$, and $\psi:B_2\rightarrow B_1$ that satisfy
		$$\phi(a)\, E_2\, b \implies a \, E_1\, \psi(b),$$
		for all $a \in A_1$, $b \in B_2$. 
	\end{defin}
	In the spirit of Propositions 2.8, and 4.9 in \cite{mhd}, we obtain:
	\begin{thm}\label{tukey}
		Let $\ABEq$, $\ABEqq$ be invariants, and let $(\phi,\psi)$ be a Tukey connection from $\ABEq$ to $\ABEqq$, where $\phi$ is definable over $N$. Then the following implications hold
		$$\diamondsuit_{\operatorname{STR}}(\ABEqq,N) 
		\implies \diamondsuit_{\operatorname{STR}}(\ABEq,N),$$
		$$\diamondsuit_{\operatorname{EL}}(\ABEqq,N) 
		\implies \diamondsuit_{\operatorname{EL}}(\ABEq,N).$$
	\end{thm}
	\begin{proof}
		We give a proof for the first implication, for the other one can be proved essentially the same way. Fix $\mathcal{H} \subseteq \Hcountable$, and let $\mathcal{B}=\langle b_\alpha|\; \alpha<\omega_1\rangle \subseteq B_2$  be such that $\good$ has a winning strategy in $\mathcal{G}_{\ABEqq}(\mathcal{H},\mathcal{B})$. Let $\mathcal{B}^*=\langle \psi(b_\alpha)|\; \alpha<\omega_1\rangle $. We will check that essentially the very same strategy is winning in $\mathcal{G}_{\ABEq}(\mathcal{H},\mathcal{B}^*)$. \emph{Essentially} means that we want to play models that are closed under $\phi$, and for this reason they must contain all possible parameters from the definition of $\phi$. This can be obviously ensured, for example by pretending that $\bad$ played these parameters in his first move. Now, since $\phi$ is definable, we have
		$$\phi[A_1\cap M_{\alpha+1}] \subseteq A_2 \cap M_{\alpha+1}$$
		at every step of the game. For this reason 
		$$(A_2 \cap M_{\alpha+1}) \; E_2 \; b_\alpha \implies \phi[A_1 \cap M_{\alpha+1}]\; E_2\; b_\alpha \implies (A_1 \cap M_{\alpha+1}) \; E_1 \; \psi(b_\alpha).$$
		This shows that the strategy is winning in $\mathcal{G}_{\ABEq}(\mathcal{H},\mathcal{B}^*)$, and concludes the proof. \qedhere
	\end{proof}

	\section{Forcing $\diamondsuit_{\operatorname{EL}}\ABE$, and $\diamondsuit_{\operatorname{STR}}\ABE$}
	
	The axioms of the form $\diamondABE$ or $\diamondelABEN$ can be forced either with finite or countable conditions. Forcing $\diamondABE$ with finite supports was systematically studied by H. Minami in \cite{minami}, and \cite{cichon}. On the other hand, the authors of \cite{mhd} give a general method for constructing models of $\diamondABE$ using countable supports. In both cases the resulting models will actually satisfy $\diamondABE$ for much bigger class of functions $F$, than merely Borel. This was relevant in \cite{idealindependent}, where the authors apply $\diamondABE$ for $F$ from $L(\mathbb{R})$. We give a sample of different arguments, to give the reader the flavour of ideas that can be applied in our setting.
	
	\subsection{By ccc forcing}
	By Theorem III.2 from \cite{hrusak}, forcing with a measure algebra over a model of $\diamondsuit$ preserves certain weaker form of $\diamondsuit(\mathfrak{d})$. Using the developed machinery we can prove a somewhat stronger result.
	\begin{thm} \label{thm6}
		For any $\kappa\ge \omega$, forcing with $\mathbb{B}_\kappa$ over a model of $\diamondsuit_{\operatorname{EL}}(\mathfrak{d}, (\Hcountable; \in))$ gives a model of $\diamondsuit_{\operatorname{EL}}(\mathfrak{d}, (\Hcountable; \in))$. In particular, it gives a model of $\diamondsuit(\mathfrak d)$.
	\end{thm}
	\begin{proof}
		
		Assume that the ground model satisfies $\diamondsuit_{\operatorname{EL}}(\mathfrak{d}, (\Hcountable; \in))$. If $\mathbb{B}_\kappa$ adds a counterexample to $\diamondsuit_{\operatorname{EL}}(\mathfrak{d}, (\Hcountable; \in))$, then so does its fragment of the form $\mathbb{B}_S$, for $|S|\le \omega_1$. Therefore we can assume, without loss of generality, that $\kappa=\omega_1$. The case $\kappa=\omega$ is similar, and left to the reader.
		
		Assume $G\in 2^{\omega_1}$ is a generic filter, and let $G_\beta:= G\restriction \beta$, for $\beta<\omega_1$. For any $x \in \Hcountable^{\mathbb{V}[G]}$ there is $\beta<\omega_1$ such that $\mathbb V[G] \models \dot{x}[G]=x$, where $\dot{x}$ is a $\mathbb{B}_\beta$-name.
		
		Let $\mathcal H \subseteq \Hcountable^{\mathbb{V}[G]}$ be arbitrary. Then, let us fix a sequence of names $\mathcal H^*= \langle \dot{h}_\alpha \mid \alpha<\omega_1 \rangle$ for elements of $\mathcal H$. We can also assume that every $\dot{h}_\alpha$ is a $\mathbb B_{\beta(\alpha)}$-name, for some countable ordinal $\beta(\alpha)$.
		
		Let $\langle r_\alpha \mid \alpha<\omega_1 \rangle \subseteq \omega^\omega$ be a sequence of reals witnessing $\diamondsuit_{\operatorname{EL}}(\mathfrak{d}, (\Hcountable; \in))$ for $\mathcal H^*$.
		
		Pick an arbitrary $X \in \mathcal P(\omega_1)\cap \mathbb V[G]$, and a $\mathbb B_{\omega_1}$-name $\dot{X}$ for it. Let $\dot{C}$ be a name for a (ground model) club subset of $\omega_1$. For every $\alpha<\omega_1$, we declare $\ast_\alpha$ to be a name for $X\cap \alpha$, such that moreover $\ast_\beta$ is always definable (in $(\Hcountable; \in)$) from $\langle \ast_\alpha \mid \alpha<\beta \rangle$. 
		
		We apply $\diamondsuit_{\operatorname{EL}}(\mathfrak{d}, (\Hcountable; \in))$ for this sequence. We pick $\beta<\omega_1$ large enough so that:
		\begin{enumerate}
			\item $G_\beta$ decides the club $\dot{C}$,
			\item $\beta \in C$,
			\item $G_\beta$ decides $\dot{X}\cap \beta$,
			\item there exists a countable $M\prec (\Hcountable; \in)$ such that
			\begin{itemize}
				\item $\ast_\beta, \; \langle \dot{h}_\alpha \mid \alpha\leq \beta\rangle \in M$,
				\item $r_\beta$ is dominating over $M$.
			\end{itemize}
		\end{enumerate}
		
		The points 1.--3. are satisfied by a club of ordinals $\beta$. For some of them, 4. will hold by the virtue of $\diamondsuit_{\operatorname{EL}}(\mathfrak{d}, (\Hcountable; \in))$. The required guessing feature in $V[G]$ will be witnessed by $M[G\cap M]=M[G_{\omega_1 \cap M}]$. The fact that $G\cap M$ is generic over $M$ is a standard consequence of the ccc. Recall that $\ast_\beta$ is a name for $X\cap \beta$. It follows that $X\cap \beta \in M[G\cap M]$, and for the same reason, $\langle h_\alpha \mid \alpha \leq \beta \rangle \in M[G\cap M]$. We need to check that $(M[G \cap M]; \in) \prec (\Hcountable; \in)^{\mathbb V[G]}.$
		
		For the verification of Tarski-Vaught critierion, assume that
		
		$$\Hcountable^{\mathbb{V}[G]}\models \exists \, y \; \phi(m,y),$$
		for some $m\in M[G\cap M]$. Fix $m$ and $y$ as above. There exist $\mathbb{B}_\omega$-names $\dot{m}$ and $\dot{y}$ for which we have
		$$\Hcountable^{\mathbb{V}}\models \; \mathbb{B}_\omega \Vdash \phi(\dot{m},\dot{y}),$$
		and $\dot{m} \in M$. Since $M\prec \Hcountable^{\mathbb{V}}$, we can assume that $\dot{y} \in M$ as well, and so $y=\dot{y}[G\cap M] \in M[G\cap M]$. It follows that 
		$$\Hcountable^{\mathbb{V}[G]}\models \phi(m,y).$$
		This shows that $M\prec \Hcountable^{\mathbb{V}[G]}$.
		
		This concludes the proof: whenever a real is dominating over $M$, it is also dominating over $M[G\cap M]$, and stationary sets remain stationary in the extension by the ccc. \qedhere
		
	\end{proof}
	\begin{remark}
		A similar argument shows that forcing with the Cohen forcing $\mathbb{C}_\kappa$ over a model of $\diamondsuit_{\operatorname{EL}}(\mathfrak{d}, (\Hcountable; \in))$ preserves $\diamondsuit_{\operatorname{EL}}(\mathfrak{b}, (\Hcountable; \in))$. 
	\end{remark}
	The next theorem is basically a reformulation of an analogous result from \cite{mhd}, which in turn is a special case of Theorem 2.12 from \cite{cichon}.
	\begin{thm}
		$\mathbb{C}_{\omega_1}$, and $\mathbb{B}_{\omega_1}$ force $\diamondstarstrnonmeagre$ and $\diamondstarstrnonnull$ respectively.
	\end{thm}
	
	\begin{proof}
		We prove the theorem for $\mathbb{C}_{\omega_1}$, the other case is similar. Let $\mathcal{H}^*=\langle \dot{h}_\alpha \mid \alpha<\omega_1\rangle$ be a sequence of $\mathbb{C}_{\omega_1}$-names for parameters in $\Hcountable$. Let $G_\alpha = G \cap \mathbb{C}_\alpha$, where $G \subseteq \omega_1 $ is $\mathbb{C}_{\omega_1}$-generic. We pick an increasing function $\phi:\omega_1 \rightarrow \omega_1$ such that $G_{\phi(\alpha)}$ decides the value of $\dot{h}_\alpha[G]$. Finally, let $c_\alpha=G\cap [\phi(\alpha),\phi(\alpha)+\omega)$. We claim that with $\mathcal{B}=\langle c_\alpha \mid \alpha<\omega_1\rangle$, $\good$ has a winning strategy. 
		
		The winning strategy is simple: $\good$ chooses $M_{\alpha+1}$ so that $M_{\alpha+1}$ contains all necessary content (that is, some parameter $h_\alpha$, initial fragment $X \cap \alpha$ of some $X\subseteq \omega_1$, and the set of previous moves $\{M_\beta|\; \beta<\alpha\}$), and moreover
		$$(M_{\alpha+1}; \in) \prec (\Hcountable; \in)^{\mathbb{V}[G_\gamma]},$$
		for some $\gamma$ as small as possible.
		We will show that, on a club set, such $\gamma$ is at most $\phi(\alpha)$. Since then $c_\alpha$ is a Cohen real over $M_{\alpha+1}$, this will finish the proof (except for checking the elementarity).
		
		Let $\dot{X}$ be a $\mathbb{C}_{\omega_1}$-name for a subset of $\omega_1$, and let $\psi:\omega_1 \rightarrow \omega_1$ be a continuous function such that
		for all $\alpha<\omega_1$, $G_{\psi(\alpha)}$ decides the value of $\dot{X}\cap \alpha$. A simple inductive argument shows that $\good$ can always choose the model $M_{\alpha+1}$, so that $M_{\alpha+1} \prec \Hcountable^{\mathbb{V}[G_{\phi(\alpha)+\psi(\alpha)}]}$. But for club many $\alpha$ we have $\psi(\alpha)=\alpha$, and so $M_{\alpha+1}\prec \Hcountable^{\mathbb{V}[G_{\phi(\alpha)}]}$. This gives us $M_{\alpha+1}\prec \Hcountable^{\mathbb{V}[G]}$, by the virtue of the next Lemma. \qedhere
	\end{proof}
	
	\begin{lem}\label{elementarity}
		Assume $G\subseteq \omega_1$ is generic for either $\mathbb C_{\omega_1}$ or $\mathbb B_{\omega_1}$, and let $G_\alpha = G\cap \alpha$. Then for any infinite $\alpha<\omega_1$, $$(\Hcountable;\in)^{\mathbb{V}[G_\alpha]}\prec (\Hcountable; \in)^{\mathbb{V}[G]}.$$
	\end{lem}
	\begin{proof}
		$$\Hcountable [G] \models \quad \phi(x_0,\ldots,x_n,y),$$
		where $x_0,\ldots,x_n \in \mathbb{V}[G_\alpha]$. In $\Hcountable^{\mathbb{V}}$ we have corresponding $\mathbb{C}_\omega$-names $\dot{x}_0,\ldots,\dot{x}_n,\dot{y}$, such that for a certain $r \in 2^\omega\cap\mathbb{V}[G_\alpha]$
		$$\mathbb{V}[r]\models \quad \dot{x}_0[r]=x_0,\ldots,\dot{x}_n[r]=x_n,$$
		and
		$$\Hcountable \models \quad \mathbb{C}_\omega\Vdash \phi(\dot{x}_0,\ldots,\dot{x}_n,\dot{y}).$$
		Therefore
		$$\Hcountable[G_\alpha]\models \quad \phi(\dot{x}_0[r],\ldots,\dot{x}_n[r],\dot{y}[r]).$$
		In particular $\dot{y}[r] \in \mathbb{V}[G_\alpha]$, and this concludes the verification of elementarity. \qedhere
	\end{proof}

	\subsection{By proper forcing}
	
	We give a version of the result from \cite{mhd}.
	
	\begin{thm}[cf. Thm. 6.6, \cite{mhd}]
		
		Let $\{\mathbb{P}_\alpha\ast\dot{\mathbb{Q}}_\alpha|\; \alpha<\omega_2\}$ be a countable support iteration of Borel forcing notions such that for each $\alpha<\omega_2$,
		$$\mathcal{P}_\alpha \Vdash ``\dot{\mathbb{Q}}_\alpha\text{ is equivalent to } \mathcal{P}(2)^+\times \dot{\mathbb{Q}}_\alpha".$$
		Assume that $\mathbb{P}_{\omega_2}$ is proper, and
		$$\mathbb{P}_{\omega_2}\Vdash \langle[A]^\omega,B,E\rangle \le \omega_1$$
		for some invariant $\ABE$ definable over $(\Hcountable; \in)$.
		Then $\mathbb{P}_{\omega_2} \Vdash \diamondstrABEcountable$.
	\end{thm}
	
	The proof of Thm. 6.6 from \cite{mhd} gives almost that. The only place we need some improvement is Lemma 6.12.
	
	\begin{lem}[cf. Lemma 6.12, \cite{mhd}]
		Let $\ABE$ be an invariant definable over $(\Hcountable; \in)$, such that $\langle A,B,E \rangle \le \omega_1$. Assume that $\mathbb{T}$ is a tree such that
		\begin{itemize}
			\item $\mathbb{T}$ is $\omega$-distributive,
			\item For every $t \in \mathbb{T}$ there exists an uncountable antichain below $t$.
		\end{itemize}
	\end{lem}
	\begin{proof}
		Just follow the proof of Lemma 6.12 in \cite{mhd}, for an invariant $([A]^\omega,B,E)$, and a function $F:2^{<\omega_1}\rightarrow [A]^\omega$, which is definable over $(\Hcountable;\in,\prec^*,\theta)$, where $\theta$ is a truth predicate for $(\Hcountable; \in, \prec^*)$. The result now follows from Theorem \ref{mainthm2}. \qedhere
	\end{proof}
	
	\section{Applications}
	\subsection{Suslin tree}
	Let us review some standard definitions. A nonempty set $T\subseteq 2^{<\gamma}$ is a \emph{tree}, if it is closed for taking initial segments. Recall that $s\sqsubset t$ means that $s$ is an initial segment of $t$. We denote as $|t|$ the \emph{length} of $t$, i.e. the order-type its domain. For $\alpha<\gamma$, we define the \emph{$\alpha$-th level} of a tree $T$:
	$$T_\alpha=\{t \in T|\ |t|=\alpha\}.$$
	The \emph{restriction} of $T$ to $t \in T$ is the set of all extensions of $t$:
	$$T\restriction t =\{s \in T|\; t \sqsubset s\}.$$
	\begin{defin}
		For a countable, limit ordinal $\alpha$, a tree $T\subseteq 2^{<\alpha}$ is
		\begin{itemize}
			\item \emph{ever-branching} if 
			$$\forall\; t \in T \quad \exists \; t'\sqsupset t \quad t'^\frown0,t'^\frown1 \in T,$$
			\item \emph{normal} if 
			$$\forall\; t \in T \quad \forall\; \beta <\alpha \quad  \exists \;t'\sqsupset t \quad |t'|>\beta.$$
		\end{itemize}
	\end{defin}

	\begin{thm}[ cf. Thm. 3.1, \cite{mhd}]
		If $\diamondelnonmeagreH$ holds, then there exists a Suslin tree.
	\end{thm}
	\begin{proof}
		Let $\langle c_\alpha \mid \alpha<\omega_1\rangle$ be a sequence witnessing $\diamondelnonmeagreH$ for $\mathcal{H}=\emptyset$. Each $c_\alpha \in 2^\omega$ induces a filter in the partial order $2^{<\omega}$, so it uniquely extends to a filter in the algebra $\overline{\mathbb{C}}_\omega$ -- the completion of $2^{<\omega}$. Let $G_\alpha$ denote this extension.\\
		
		We will build a normal, ever-branching tree $T$ level by level, starting from $T_0=\{\emptyset\}$. If $\alpha=\beta+1$, we just extend each element of $T_\beta$, i.e. $$T_\alpha=\{t^\frown0,t^\frown1|\; t \in T_\beta\}.$$
		If $\alpha$ is limit, we extend each $t \in T_{<\alpha}$ to a branch $b_t\in 2^\alpha$ accordingly with $c_\alpha$: let $i_t$ be the $\prec^*$-least dense embedding of $T\restriction t$ into $\overline{\mathbb{C}}_\omega$, and put
		$$b_t=i_t^{-1}[G_\alpha].$$
		
		Finally, we put $T=\bigcup\{T_\alpha|\; \alpha<\omega_1\}$. We will check that $T$ has no uncountable antichain. Let $\langle \ast_\alpha \mid \alpha<\omega_1\rangle$ be an enumeration enumeration of a maximal antichain in $T$, such that $\ast_\beta = \langle \ast_\alpha 
		\mid \alpha<\beta \rangle$, for every limit $\beta<\omega_1$. A standard argument shows that for club-many $\alpha$, $\ast_\alpha$ is a maximal antichain in $T_{<\alpha}$, therefore we can find such $\alpha$, together with a model $M_{\alpha}$, which satisfies:
		\begin{itemize}
			\item $\ast_\alpha,\; T_{<\alpha} \in M_{\alpha}$,
			\item $c_\alpha$ is a Cohen real over $M_{\alpha}$.
		\end{itemize}
		In this case also $G_\alpha$ is $\overline{\mathbb{C}}_\omega$-generic over $M_{\alpha}$, and since each embedding $i_t$, for $t \in T_{<\alpha}$, is in $M_{\alpha}$, also $b_t$ is  a generic branch through $T_{<\alpha}$, over $M_{\alpha}$. We aim to show that $\ast_\alpha$ is a maximal antichain in $T_{<\alpha}$. Fix $t \in T_{<\alpha}$ if $t$ has an initial segment in $\ast_{\alpha}$, there is nothing to prove. Otherwise, consider the set
		$$\{t'\sqsupset t \mid \exists \; \beta<\alpha \quad t'\cap \beta \in \ast_{\alpha} \}.$$
		Since $\ast_\alpha$ is maximal in $T_{<\alpha}$, this set is dense above $t$ in $T_{<\alpha}$. Since the branch $b_t$ was generic over $M_{\alpha}$, $b_t$ must have an initial fragment in $\ast_{\alpha}$. This concludes the proof. \qedhere
	\end{proof}
	\subsection{$\clubsuit_{\operatorname{AD}}(\omega_1,1,\operatorname{DIAG})$}
	The result we present was originally proved in \cite{mhd}. We recast in the language of the axioms $\clubsuit_{\operatorname{AD}}(\kappa,\mathcal{S},\lambda)$, that were introduced in \cite{clubad}, and later applied in \cite{clubad2}.
	
	\begin{defin}
		The principle $\clubsuit_{\operatorname{AD}}(\omega_1,1,\operatorname{DIAG})$
		asserts the existence of a sequence $\langle A_\alpha|\; \alpha<\omega_1\rangle$ such that:
		\begin{enumerate}
			\item For each limit $\alpha<\omega_1$, $A_\alpha \subseteq \alpha$ is cofinal;
			\item For any $\alpha<\beta<\omega_1$, $\sup(A_\alpha \cap A_\beta)<\alpha$;
			\item For each sequence $\langle B_\gamma|\; \gamma<\omega_1\rangle $ consisting of cofinal subsets of $\omega_1$, the set
			$$\{\alpha<\omega_1 \mid \forall{\gamma<\alpha}\quad B_\gamma\cap A_\alpha\text{ is cofinal in }\alpha\}$$
			is stationary.
		\end{enumerate}
	\end{defin}
	\begin{thm}[cf. Thm. 5.5, \cite{mhd}]
		$\diamondelbH$ implies $\clubsuit_{\operatorname{AD}}(\omega_1,1,\operatorname{DIAG})$, with each $A_\alpha$ having the order-type $\omega$.
	\end{thm}
	
	\begin{proof}
		We fix a $\diamondelbH$-sequence $\langle d_\alpha|\; \alpha<\omega_1\rangle$, subject to a sequence of initial parameters
		$$\mathcal{H}=\{\{B_n^\alpha\}_{n<\omega} \mid \omega \cdot \omega \leq \alpha<\omega_1\},$$
		where for every limit $\alpha \ge \omega \cdot \omega$, $\{B_n^\alpha\}_{n<\omega}$ is a fixed partition of $\alpha$ into infinite sets, such that whenever $n<m$, each element of $B_n^\alpha$ is below each element of $B_m^\alpha$. 
		
		Informally, we pick to $A_\alpha$ the first $d_\alpha(n)$-many elements from each set $B_n^\alpha$. More precisely, we enumerate bijectively each of the sets
		$$B_n^\alpha=\{b^\alpha_{n,i}\}_{i<\omega},$$
		and set
		$$A_\alpha=\bigcup_{n<\omega}\{b^\alpha_{n,0},\ldots,b^\alpha_{n,d_\alpha(n)}\}.$$
		Why does that work? Fix a family $\{B_\gamma|\; \gamma<\omega_1\}$ consisting of cofinal subsets of $\omega_1$. We associate with it a continuous sequence of hereditarily countable sets $\ast_\alpha$, by setting
		$$\ast_\alpha=\{B_\delta\cap \alpha|\; \delta<\alpha\}.$$
		There exists a stationary set of $\alpha$, for which $d_\alpha$ is not dominated by any real from some model $M_{\alpha}$ containing $\ast_\alpha=\langle \ast_\beta|\; \beta<\alpha\rangle $, and $\{B_n^\alpha\}_{n<\omega}$. It is clear that for each $\beta<\alpha$, $B_\beta \cap \alpha \in M_{\alpha}$. A standard closing-off argument shows that for club-many $\alpha$'s, each of the sets $\{B_\delta\cap \alpha\}_{\delta<\alpha}$ is cofinal in $\alpha$. We work with fixed $\beta<\alpha$, such that
		\begin{itemize}
			\item $B_\beta\cap \alpha$ is cofinal in $\alpha$,
			\item $B_\beta \cap \alpha \in M_{\alpha}$,
			\item $d_\alpha$ is an unbounded real over $M_{\alpha}$.
		\end{itemize}
		Let $A=\{n<\omega|\; B_\beta \cap B^\alpha_n \neq \emptyset\}$. Obviously, since $B_\beta \cap \alpha$ is cofinal, this set is infinite and belongs to $M_{\alpha}$.
		
		Let $A(n)$ denote the $n$-th element of $A$,  and set
		$$r(n)=\min\{i<\omega|\; b^\alpha_{A(n),i} \in B_\beta\}.$$
		Given that $r \in \omega^\omega\cap M_{\alpha}$, the inequality $r(n)<d_{\alpha}(A(n))$ holds for infinitely many $n<\omega$. It's now easy to see that $A_\alpha\cap B_\beta$ is infinite. \qedhere
	\end{proof}
	
	\subsection{$\clubsuit_{\operatorname{AD}}(\omega_1,\omega,\operatorname{DIAG})$}
	We prove another result in the same spirit.
	\begin{defin}
		The principle $\clubsuit_{\operatorname{AD}}(\omega_1,\omega,\operatorname{DIAG})$ asserts the existence of a sequence $\{B^n_\alpha|\; \alpha<\omega_1, \; n<\omega\}$ such that:
		\begin{enumerate}
			\item For each limit $\alpha<\omega_1$, and each $n<\omega$, $B^n_\alpha \subseteq \alpha$ is cofinal;
			\item For any $\alpha<\beta<\omega_1$, $m,n<\omega$, $\sup(B^n_\alpha \cap B^m_\beta)<\alpha$;
			\item For each sequence $\{X_\gamma|\; \gamma<\omega_1\}$ consisting of cofinal subsets of $\omega_1$, the set
			$$\{\alpha<\omega_1|\; \forall\,{\gamma<\alpha}\; \forall\, n<\omega\quad X_\gamma\cap B^n_\alpha \text{ is cofinal in }\alpha\}$$
			is stationary.
		\end{enumerate}
	\end{defin}

	\begin{lem}\label{dominating}
		Suppose $r,d \in \omega^\omega$ are strictly increasing, and $r<^* d$. Let $f\in \omega^\omega$ be defined by
		$$f(n)=d\circ\ldots\circ d(n),$$
		where the composition is taken $n$ times. Then for all but finitely many $n<\omega$ we have
		$$[f(n),f(n+1))\cap \operatorname{rg}(r)\neq \emptyset.$$
	\end{lem} 
	\begin{proof}
		Without loss of generality $r(k)>k$ for all $k$. Suppose $r(k)<f(n)$. We must show that $r(k+1)<f(n+1)$. But we have
		$$r(k+1)\le d(k+1)\le d(r(k))\le d(f(n))<f(n+1).\qedhere$$ 
	\end{proof}
	
	Recall that $\diamondstareldH$ is $\diamondeldH$ that guesses on a club, instead of a stationary set.
	
	\begin{thm}
		$\diamondstareldH$ implies $\clubsuit_{\operatorname{AD}}(\omega_1,\omega,\operatorname{DIAG})$, with each $A_\alpha$ having the order-type $\omega$.
	\end{thm}
	\begin{proof}
		By the previous theorem we can fix a sequence $\langle A_\alpha|\; \alpha\in \lim(\omega_1)\rangle$ witnessing $\clubsuit_{\operatorname{AD}}(\omega_1,1,\operatorname{DIAG})$, each $A_\alpha$ having the order-type $\omega$, and all intervals of the form $[A_\alpha(n),A_\alpha(n+1))$ being infinite. Let $\langle d_\xi|\; \xi \in \lim(\omega_1)\rangle$ be a sequence of (strictly increasing) reals witnessing $\diamondstareldH$ for initial parameters $h_\alpha=A_\alpha$. Let $\phi:\omega \times \omega \rightarrow \omega$ be a definable bijection, and define
		$$B^n_\alpha=\{A_\alpha(i)|\; i \in \bigcup_{k<\omega} I^\alpha_{\phi(n,k)} \},$$
		where 
		$$I^\alpha_k=[d^k_\alpha(k),d^{k+1}_\alpha(k+1) ),$$
		for each $k<\omega$ (here $d^n$ denote the $n$-th iteration of $d$).
		
		To see that this works, consider a family $\{X_\alpha|\; \alpha<\omega_1\}$ of cofinal subsets of $\omega_1$, and put 
		$$\ast_\delta=\{X_\beta \cap \alpha|\; \alpha<\beta<\delta \},$$ for all $\delta \in \omega_1$.
		Clearly this sequence is continuous. By $\diamondstareldH$ there exists a club set of ordinals $\xi < \omega_1$, for which we can choose the corresponding elementary submodels $\langle M_\xi$, satisfying:
		\begin{itemize}
			\item $\{X_\beta|\; \beta<\xi\}\in M_{\xi}$,
			\item $d_\xi$ is a dominating real over $M_{\xi}$,
			\item $\sup(A_\xi \cap X_\beta)=\xi$, for all $\beta<\xi$.
		\end{itemize}
		Notice that here we used the feature of guessing on a club: the sets 
		$$\{\xi<\omega_1|\; \forall \; \beta<\xi\quad A_\xi \cap X_\beta\text{ is cofinal} \},$$
		and
		$$\{\xi<\omega_1|\; d_\xi \text{ is dominating over }M_{\xi}\}$$
		need to intersect on a club. Now fix $\beta<\xi$, where $\xi$ is as above, and let 
		$$r=\langle i<\omega|\; A_\xi(i)\in X_\beta \rangle.$$
		By Lemma \ref{dominating} the set $r$ is infinite, and $r \in M_{\xi}$, therefore $r \cap I^\xi_k \neq \emptyset$ for all but finitely many $k<\omega$. It follows that for any $n<\omega$, $r \cap I^\xi_{\phi(n,k)}\neq \emptyset$, for all but finitely many $k<\omega$. But $i \in r\cap I^\xi_{\phi(n,k)}$ implies $A_\xi(i) \in X_\beta \cap B^n_\xi$. It follows that $X_\beta \cap B^n_\xi$ is infinite, and therefore cofinal in $\xi$. \qedhere
	\end{proof}
	
	\subsection{$\aleph_1$-generated P-point}
	
	\begin{thm}[cf. Thm. 7.8, \cite{mhd}]
		If $\diamondsuit_{\operatorname{EL}}(\mathfrak{r},(\Hcountable; \in))$ holds, then there exists an $\aleph_1$-generated P-point. 
	\end{thm}
	\begin{proof}
		Let $\langle r_\alpha|\; \alpha<\omega_1\rangle$ be a $\diamondsuit_{\operatorname{EL}}(\mathfrak{r},(\Hcountable; \in))$-sequence, for the sequence of initial parameters consisting of bijections $h_\alpha:\alpha\rightarrow \omega$. Let $\langle r^*_\alpha|\; \alpha<\omega_1\rangle$ be a sequence of reals (more precisely, infinite subsets of $\omega$) defined inductively by conditions:
		\begin{enumerate}
			\item $r^*_0=r_0$,
			\item If $r^*_\alpha$ is known for $\alpha<\beta$, then define $f_{<\beta}:\omega \rightarrow \omega$ by 
			$$f_{<\beta}(n)=\min(\bigcap_{k<n+1}r^*_{h_\beta(k)}\setminus f_{<\beta}(n-1)),$$
			for all $n>0$ (let $f(0)=0$). Then the image of $f_{<\beta}$ is a pseudo-intersection of the family $\langle r^*_\alpha|\; \alpha<\beta\}$, that is definable from $h_\beta,\{r^*_\alpha|\; \alpha<\beta\rangle $. Finally, put
			$$r^*_\beta=\{f_{<\beta}(n)|\; n\in r_\beta \}.$$
		\end{enumerate}
		Let $U=\{a \in [\omega]^\omega|\; \exists \; \beta<\omega_1 \quad r^*_\beta \subseteq^* a\}$. Clearly $U$ is a P-filter. Let us verify that it is also an ultrafilter. Let $a \in [\omega]^\omega$ be arbitrary. There exists cofinally many ordinals $\beta<\omega_1$, for which we can pick an elementary submodel $M_\beta \prec\Hcountable$, such that:
		
		\begin{itemize}
			\item $r_\beta$ is reaping over $M_{\beta}$,
			\item $a, \langle r_\alpha|\; \alpha<\beta\rangle , \langle h_\alpha|\; \alpha\le \beta\rangle \in M_{\beta}$.
		\end{itemize} 
		Note that in this case, also $\langle r^*_\alpha|\; \alpha<\beta\rangle \in M_{\beta}$. Fix such ordinal $\beta$. Let $b\subseteq \omega$ be such that $f_{<\beta}[b]=f_{<\beta}[\omega]\setminus a$. Notice that $b\in M_{\beta}$. Given that $r_\beta$ is reaping over $M_{\beta}$, either $r_\beta \subseteq^* b$, or $|r_\beta \cap b|<\omega$. 
		
		In the first case, we got 
		$$r^*_\beta=f_{<\beta}[r_\beta]\subseteq^*f_{<\beta}[b]=f_{<\beta}[\omega]\setminus a \subseteq \omega \setminus a.$$
		This ensures that $\omega \setminus a \in U$.
		
		In the second case, we got
		$$|r^*_\beta \cap f_{<\beta}[b]|<\omega,$$
		and in turn
		$$|r^*_\beta \cap(f_{<\beta}[\omega]\setminus a)|<\omega.$$
		It follows that $r^*_\beta \subseteq^* a$, and so $a \in U$. \qedhere
	\end{proof}
	
	\subsection{$\clubsuit_{\operatorname{w}^2}$}
	
	The axiom $\clubsuit_{\operatorname{w}^2}$ is a weakening of $\clubsuit$ studied initially in \cite{sticks1}, and \cite{sticks2}. 
	\begin{defin}
		The principle $\clubsuit_{\operatorname{w}^2}$ asserts the existence of a sequence $\langle C_\alpha|\; \alpha \in \lim(\omega_1)\rangle$,
		where each $C_\alpha$ is a cofinal subset of $\alpha$, and for every cofinal $X \subseteq \omega_1$, the set
		$$\{\alpha<\omega_1 \mid |C_\alpha \setminus X|<\omega \text { or } |C_\alpha \cap X|<\omega\}$$
		is stationary.
	\end{defin}
	\begin{thm}[cf. Lemma 2.4, \cite{milden}]
		$\diamondsuit_{\operatorname{EL}}(\mathfrak{r},(\Hcountable; \in)) \implies \clubsuit_{\operatorname{w}^2}$
	\end{thm}
	
	\begin{proof}
		Let $\{r_\alpha|\; \alpha \in \omega_1 \}\subseteq [\omega]^\omega$ be a sequence witnessing $\diamondelrH$ for the parameters $h_\alpha=\sigma_\alpha$, where for each limit $\alpha < \omega_1$ $\sigma_\alpha:\omega \rightarrow \alpha$ is an increasing function onto a cofinal subset. Now, for each limit $\alpha$ we put $C_\alpha=\sigma_\alpha[r_\alpha]$. We claim that $\{C_\alpha|\; \alpha \in \lim(\omega_1)\}$ is a $\clubsuit_{\operatorname{w}^2}$-sequence.
		
		Indeed, fix $X \subseteq \omega_1$. There is a stationary set of indices $\alpha$, for which we can choose elementary submodels $M_\alpha \prec (\Hcountable; \in)$, such that:
		\begin{itemize}
			\item $X\cap \alpha, \sigma_\alpha \in M_{\alpha},$
			\item $r_\alpha$ is reaping over $M_{\alpha}$.
		\end{itemize}
		In particular, either $r_\alpha \subseteq^* \sigma_\alpha^{-1} [X\cap \alpha ]$, or $\sigma_\alpha^{-1} [X\cap \alpha ] \cap r_\alpha$ is finite. It follows that $C_\alpha = \sigma_\alpha[r_\alpha]$ is either almost contained in $X\cap \alpha$, or almost disjoint with it. \qedhere
	\end{proof}

\end{document}